\documentclass{amsart}
\usepackage{pgf,pgfarrows,pgfnodes,pgfautomata,pgfheaps,pgfshade,hyperref, amssymb}

\newcommand{\supp}{\operatorname{supp}}

\newtheorem{teo}{Theorem}[section]

\newtheorem{lemma}[teo]{Lemma}

\newtheorem{prop}[teo]{Proposition}

\newtheorem{coro}[teo]{Corollary}

\theoremstyle{definition}

\def\R{\mathcal{R}}

\newcommand{\RR}{\mathbb{R}}

\def\RR{\mathbb{R}}

\def\R{\mathcal{R}}

\theoremstyle{definition}
\newtheorem{definition}[teo]{Definition}

\def\R{{\mathbb R}}
\def\Z{{\mathbb Z}}

\def\cF{{\mathcal F}}
\def\cB{{\mathcal B}}

\def\cL{{\mathcal L}}

\def\Homeo{\operatorname{Homeo}}

\begin{document}

\title[Simplicity of homeomorphism groups]{On the simplicity of  homeomorphism groups of a tilable lamination}
\author{Jos\'e Aliste-Prieto and Samuel Petite}


\thanks{J. Aliste-Prieto acknowledges financial support from Fondecyt Iniciaci\'on 11121510 and Anillo DySyRf ACT-1103. S. Petite acknowledges financial support from the ANR SUBTILE 0879. This  work is part of the program MathAmSud  DYSTIL 12Math-02} 

\subjclass{57S05, 37C85 } \keywords{leaf preserving homeomorphisms, tilable lamination.}

\maketitle
\begin{abstract} We show that the identity component of the  group of homeomorphisms that preserve all leaves of a  ${\mathbb R}^{d}$- tilable lamination is simple.  Moreover, in the one dimensional case, we show that this group is uniformly perfect.  We obtain a similar result for a  dense subgroup  of homeomorphisms.
\end{abstract}

\maketitle \markboth{Jos\'e Aliste-Prieto and Samuel Petite}{ }

\section{Introduction}

In this paper it is shown that the connected component of the identity  of the group $\Homeo_\cL(\Omega)$ 
of all leaf-preserving homeomorphisms of a minimal tilable lamination $\Omega$ in any dimension is a simple group. 
We also prove that this group is equal to the group of homeomorphisms that are isotopic to the identity and that is open in $\Homeo_\cL(\Omega)$. 

Similar results were obtained in the 60's by G. Fisher \cite{F} for the group of all homeomorphisms of a closed topological manifold of dimension smaller or equal than three.
The algebraic simplicity for groups of homeomorphisms and diffeomorphisms 
of manifolds has been widely studied in the literature:  In 1961, R. Anderson \cite{An}, generalizing the work of G. Fisher \cite{F}, showed  the group of stable homeomorphisms of a manifold is simple. Later,   D. Epstein  \cite{E} established  sufficient  conditions  on a group of homeomorphisms,   for  the commutator subgroup to be simple. This means that a group satisfying Epstein's conditions is simple if and only if it is perfect  ({\em i.e.} its commutators subgroup is the whole group). 

It is also worth mentioning the works of M.  Herman \cite{H}, W. Thurston \cite{T} and  J. Mather \cite{Ma} who provided a nearly complete classification 
for the simplicity of diffeomorphism groups on manifolds. 

Given  a  smooth foliation $\cF$ over  a manifold $M$, T. Rybicki \cite{R} and T. Tsuboi \cite{Ts}  studied the simplicity and perfectness of the identity component of the group $G_\cF(M)$ of all leaf preserving  diffeomorphisms of $(M, \cF)$. Notice here that these groups do not satisfy Epstein's conditions.

On the other hand, tilable laminations have been recently introduced as a geometric model for the study of non-periodic tilings \cite{BBG}. They also appear as suspensions of  minimal Cantor $\Z^d$-actions, like  minimal subshifts. In addition, they include some classical laminated spaces as the dyadic solenoid. These  spaces are locally homeomorphic to the product of an open set in $\R^d$ and a Cantor set. In other words, these are laminated spaces with a Cantor transversal. They are also endowed with a natural $\R^d$-action, which we we call the translation flow. Like in the case of foliations of manifolds, groups of homeomorphisms on tilable laminations in general do not satisfy  Epstein's conditions. 


We denote by $D(\Omega)$ the group of {\em deformations}, that is, $D(\Omega)$ is the path connected component of the identity in $\Homeo(\Omega)$ endowed with  the $C^0$-topology. 
 We conjointly consider the following class of homeomorphisms, which arises  naturally in the context of non-periodic tilings, namely the group of homeomorphisms preserving the vertical structure (see the precise definition in Section \ref{sec:preliminaries}) denoted $\Homeo_{vsp}(\Omega)$. Roughly speaking, such homeomorphisms map any small Cantor transversal into a Cantor transversal. Notice that $\Homeo_{vsp}(\Omega)$ is a dense subgroup of $\Homeo_\cL(\Omega)$.  
 Let $D_{vsp}(\Omega)$ denote the path connected component of the identity in $\Homeo_{vsp}(\Omega)$.    Our aim is to show the simplicity of these groups. 

Another motivation for studying these groups comes from topological  orbit equivalence theory: two tilable laminations $\Omega_{1}, \Omega_{2}$ are orbit equivalent if there is a homeomorphism between the spaces mapping any orbit onto an orbit. Because of the totally disconnected transversal structure, $\Omega_{1}$ and $\Omega_{2}$ are orbit equivalent if and only if they are homeomorphic. A difficult result of Rubin \cite{Ru} asserts that the group  $\operatorname{Homeo}(\Omega)$ is a complete invariant of the orbit equivalence class of the lamination $\Omega$: Any algebraic group isomorphism  of these groups is induced by an homeomorphism on the topological laminations. It follows that the group $D(\Omega)$ is an invariant of flow equivalence.

For a topological group $G$, we denote by $G^0$ the connected component of the identity in $G$. 
\begin{teo}\label{teo:main} Let $\Omega$ be a minimal tilable lamination. Let $G$ be either $\Homeo_\cL(\Omega)$ or $\Homeo_{vsp}(\Omega)$. Then,
\begin{enumerate}
\item $\Homeo^0_\cL(\Omega) = D(\Omega)$ and $ \Homeo_{vsp}^0(\Omega) = D_{vsp}(\Omega)$;
\item $G^0$ is open in $G$;
\item $G^0$ is simple.
\end{enumerate}
\end{teo}    

Moreover, when the translation flow is expansive, e.g. for tiling spaces, the connected component of the identity in $\Homeo(\Omega)$ is the group of deformations.
\begin{prop}\label{prop:direct2} Let $\Omega$ be a minimal tilable lamination. If the translation flow is expansive, then the identity component $\Homeo^0(\Omega)$  is equal to $D(\Omega)$ and is open in $\Homeo(\Omega)$.
\end{prop}
The proof of Theorem \ref{teo:main} follows the same strategy as in \cite{F} for the triangulated manifolds (see \cite{B} for a recent survey). In the next section we recall basic properties of  tilable laminations and their homeomorphisms. By using a generalization of the Schoenflies Theorem due to R. Edwards and R. Kirby, we show, in Section \ref{sec:partition}, the groups under consideration satisfy the partition property (called also fragmentation property), and we prove  Proposition \ref{prop:direct2} and the items (1), (2) of Theorem \ref{teo:main}. We give in Section \ref{sec:simple} a sufficient condition for a commutator subgroup of $\Homeo_{\cL}(\Omega)$ to be simple. Next,  we prove  in Section \ref{sec:perfect}  that the groups $\Homeo_{\cL}(\Omega) $ and $\Homeo_{vsp}(\Omega)$  are perfect and we conclude the proof  of Theorem \ref{teo:main} with the main result of Section \ref{sec:simple}. In the last section, we show, for the one-dimensional case, that these groups are {\em uniformly perfect}: more precisely, any element can be written as a product of two commutators in  the group.  This last  result is similar to \cite{F2} for $C^{\infty}$ leaf preserving diffeomorphisms of $C^{\infty}$ foliations.

\section{Preliminaries}\label{sec:preliminaries}

\subsection{Background on tilable lamination}\label{sec:background}

We recall here some basic properties of tilable laminations  and we refer to \cite{BBG} for a more detailed exposition. Let $\Omega$ be a compact metric space. Assume that  there exist a cover of $\Omega$ by open sets $U_{i}$  and homeomorphisms called {\em charts} $h_{i} \colon U_{i} \to  D_{i}\times C_{i} $ where $C_{i}$ is a topological space and $D_{i}$ is an open set of $\R^{d}$. These open  sets and charts define an \emph{atlas} of a \emph{flat   lamination} if the transition maps $h_{j}\circ h_{i}^{-1}$ read on their domains of definition 
\begin{equation}
\label{transitionmaps}
h_{i,j} (t,c) = ( t+a_{i,j}, \gamma_{i,j} (c) ),
\end{equation}
where the $a_{i,j}$'s are elements of $\R^{d}$ and the maps $\gamma_{i,j}$ are continuous. Two atlases are \emph{equivalent} if their union is also an atlas.

A flat lamination  is the data of a compact metric space $\Omega$  together with an equivalence class of atlases $\cL$. A {\em  box} is the domain of a chart in the maximal atlas of $\cL$. 
For any point $x$ in a box $B$ with coordinates $( t_{x}, c_{x })$ in the chart $h$, the set $h^{-1}( D \times \{c_{x}\})$ is called the \emph{slice} and the set $h^{-1}( \{t_{x}\} \times C)$ is called the \emph{vertical} of $x$ in $B$. Since a transition map  transforms slices into slices and verticals into verticals, these definitions make sense. As usual, a \emph{leaf} of $\Omega$ is the smallest connected set that contains all the slices it intersects. From \eqref{transitionmaps}, it is clear that each leaf is a manifold with a flat Riemannian metric.

\begin{definition}
A \emph{tilable lamination} $(\Omega, \cL)$  (or a $\R^{d}$-solenoid) is a flat lamination such that
\begin{itemize}
\item  every  leaf of $\cL$ is isometric to $\R^{d}$. 
\item  There exists a transversal $\Xi$ (a compact subset of $\Omega$ such that for any leaf $L$ of $\cL$, $L \cap \Xi$ is non empty and a discrete subset with respect the manifold topology of the leaf $L$) which is a Cantor set.  
 \end{itemize}
\end{definition}
For short, we will  speak about $ \Omega$ as a tilable lamination when there is no confusion. If every leaf is dense in $\Omega$, we say that the lamination is {\em minimal}. By \eqref{transitionmaps}, the action by translations on $\R^d$ can be transported to a local action (also by translations) along the slices. In fact, these local translations induce  a continuous  and free $\R^{d}$-action $T$ over $\Omega$, see  \cite{BBG} for details. We refer to this action as the \emph{translation flow} over $\Omega$. To simplify the notations we write  $\omega - t := T(t, \omega)$ for $\omega$ in $\Omega$ and $t$ in $\R^d$. It is easy to see that the leaves of the lamination  coincide with the orbits of the translation flow. Again by \eqref{transitionmaps}, the canonical orientation on $\R^d$ induce an orientation on each leaf of $\Omega$. Given a box $B$ that reads $h^{-1}(D \times C)$ in a chart $h$, by identifying a vertical in $B$  with the Cantor set $C$, we can write $B = T(C, D) = C - D$, thus avoiding the explicit reference to the chart $h$.

Basic examples of  minimal tilable laminations are given by the suspensions   of  minimal $\Z^d$ action on a Cantor set with locally constant ceiling functions. The  tilable lamination structure also appears in the  dynamical systems associated to  non-periodic repetitive tilings and Delone sets of the Euclidean space, see \cite{BBG}.  In  these examples, the translation flow is {\em expansive} in the following sense (see \cite{PFK}).
\begin{definition}\label{def:expansive}
Let $\eta>0$. The translation flow of a tilable lamination $\Omega$ is said to be $\eta$-expansive if when one has points $x,y \in \Omega$ and a homeomorphism $h \colon \R^d \to \R^d$  satisfying  $h(0)=0$ and 
$d(x - t, y - h(t)) < \eta$ for all $t\in \R^d$, then there must exist $t_{0}\in B_{\eta}(0)$ such that $x - t_{0} = y$.

The translation flow is said {\em expansive}, if it is $\eta$-expansive for some constant $\eta$. 
\end{definition}
This last property will allow us to show, in the next section, that any homeomorphism that is close enough to the identity must by leaf-preserving. 

A box in $\Omega$ is said to be \emph{internal} if its closure is included in another box of $\Omega$. In all the rest of the paper, any box will be supposed to internal.  An internal box $B$ is said to be \emph{of ball type} if it can be written as $B = C - D$, where $D$ is an open ball in $\RR^d$. Instead, if $D$ is a $d$-cube $(a_{1}, b_{1}) \times \cdots \times (a_{d},b_{d})$ in  $\R^{d}$, then $B$ is said to be a box of \emph{cubic} type.  In this case, and if $f$ is a $\ell$-face ($0 \le \ell \le d$) of the cube $D$, then the set $C - f$ 
is said to be an $\ell$-\emph{vertical boundary} of $B$. A \emph{box cover} of $\Omega$ is a cover $\{B_i\}_i$ of $\Omega$, where each $B_i$ is a box. Box covers of ball type and cubic type are defined in the same way. 

\begin{definition}
A collection of boxes $\cB=\{B_i\}_{i=0}^{t}$ in  $\Omega$ is  a \emph{box decomposition}, if the following assertions hold:
\begin{enumerate}
\item the $B_{i}$'s are pairwise disjoint,
\item the closures of the $B_{i}$'s form a  cover of $\Omega$.
\end{enumerate}
Also, if the boxes $B_i$ are of cubic type, then $\cB$ is a box decomposition of cubic type.
\end{definition}
Box decompositions were introduced in \cite{BBG} as a tool in the study of tilable laminations. The key lemma (see bellow) asserts that any box cover
of cubic type can be turned into a box decomposition of cubic type. It follows that every tilable lamination admits a box decomposition of cubic type. 
\begin{lemma}[\cite{BBG}]\label{lem:BBG}
Let $\Omega$ be a tilable lamination and  $\{B_i\}_{i=0}^{t}$ be a box cover of cubic type of $\Omega$. Then, there exists a box decomposition of cubic type $\cB'=\{B'_i\}_{i=0}^{n}$, such that, for all $i$, whenever $B'_{i}$ intersects $B_{j}$ for some $j$, then it is included in $B_{j}$.
\end{lemma}
The union of all the $\ell$-vertical boundaries ($0\le \ell \le d$) of all the boxes  of a box decomposition of cubic type $\cB$  is called the $\ell$-\emph{skeleton} of $\cB$. 

\subsection{Homeomorphisms of tilable laminations}\label{sec:backgroundhomeo}
Let $\Omega$ be a tilable lamination and denote by $\Homeo(\Omega)$ the set of homeomorphisms of $\Omega$. We endow it with the $C^{0}$-topology, which is induced by the distance
\[\delta(f,g) = \sup_{x \in \Omega } d(f(x), g(x))+ \sup_{x \in \Omega} d(f^{-1}(x), g^{-1}(x)), \quad  f,g \in \Homeo(\Omega).\]

The \emph{support} of a homeomorphism $f$ in $\Homeo(\Omega)$ is defined by
\[\operatorname{supp} f = \overline{\{x\in \Omega \mid f(x)\neq x\}}.\] 
It is easy to see that $\operatorname{supp} f$ is  $f$-invariant and $\supp \phi f \phi^{-1} = \phi (\supp f)$ for every $\phi \in \Homeo(\Omega)$.

Since the verticals of  a tilable lamination $\Omega$ are totally disconnected, the path-connected components coincide with the leaves of the lamination. Thus, every  element of $\Homeo(\Omega)$ maps each leaf onto a (possibly different) leaf. We define $ \Homeo_{\cL}(\Omega)$ be the group of all leaf-preserving homeomorphisms of $\Omega$. Recall that a homeomorphism $f$ of $\Omega$ is \emph{homotopic to the identity}, if there exists a continuous map $F:[0,1]\times \Omega\rightarrow \Omega$ such that $F(0,\cdot) = Id$ and $F(1,\cdot) = f$. If, in addition, $F(t,\cdot)$ is a homeomorphism of $\Omega$ for each $t\in \Omega$, then we say that $f$ is \emph{isotopic to the identity} or a {\em deformation}. The set $D(\Omega)$ denotes the group of all the deformations. Clearly, homeomorphisms
that are homotopic to the identity belong to $\Homeo_\cL(\Omega)$. If $\Omega$ is minimal, then the converse is also true. 
\begin{teo}[\cite{A},\cite{K}]
\label{teo:strhomeo}
Let $\Omega$ be a minimal tilable lamination. Then every $f \in\Homeo_\cL(\Omega)$ is homotopic to the identity. In particular, for every $f\in\Homeo_\cL(\Omega)$, there is a continuous map  $\Phi_f\colon \Omega \to \R^d$, called the \emph{displacement} of $f$, which is uniquely defined by the equation
\[ f(\omega) = \omega -\Phi_f(\omega) \quad  \text{for all }\omega \in \Omega.\] 
\end{teo}
We say the displacement of $f$ is smaller than $\varepsilon$ when $\vert \vert \Phi \vert \vert_{\infty} < \varepsilon$. 

When the translation flow is expansive, we get the following refinement.
\begin{prop}\label{lem:clopenhomeo} Let $\Omega$ be tilable lamination. Suppose its translation flow is $\eta$-expansive. Then, $\Homeo_{\cL}(\Omega)$ is open in $\Homeo(\Omega)$.
\end{prop}
\begin{proof} 
Define $B = \{f\in\Homeo(\Omega)\mid \delta(f,Id)< \eta\}$ and take any $f\in B$. Since the translation flow is free and homeomorphisms map leaves onto leaves, for every $\omega\in \Omega$
there is a continuous map $h:\R^d \to \R^d$ such that $f(\omega)- s  = f(\omega -h(s))$ for all $s \in \R^d$.

Thus, 
\[ d(f(\omega)-s, \omega -h(s)) \leq\delta(f,Id) < \eta \quad \text{for all } s\in\R^d.\]
It follows from the expansivity of $\Omega$ that there exists  a $t_{0}\in \R^d$ such that $f(\omega) = \omega - t_{0}$. Since $\omega$ was arbitrary, this means that $f$ preserves each leaf and thus $f\in  \Homeo_{\cL}(\Omega)$, which means that $B\subset\Homeo_{\cL}(\Omega)$. The fact that $\Homeo_{\cL}(\Omega)$ is open
now follows from a standard argument.
\end{proof}

\begin{coro}\label{cor:idcomponent}If the translation flow on the tilable lamination $\Omega$ is expansive, then $\Homeo^{0}(\Omega)$ is open and 
$$\Homeo^{0}(\Omega) =\Homeo_{\cL}^0(\Omega).$$
\end{coro}
\begin{proof} 
Since the connected component is the greatest connected set containing the identity,  we have $\Homeo_{\cL}^{0}(\Omega) \subset \Homeo^{0}(\Omega)$. By  Proposition \ref{lem:clopenhomeo} and the connexity property, we get $\Homeo^{0}(\Omega) \subset \Homeo_{\cL}(\Omega)$ and so $\Homeo^{0}(\Omega) \subset \Homeo_{\cL}^{0}(\Omega)$, which concludes the proof.
\end{proof}

In the context of laminations arising from the study of non-periodic tilings, an important class of homeomorphisms is given by homeomorphisms with the following property. 
\begin{definition}\label{def:lcrt}
A homeomorphism $f \in  \Homeo(\Omega)$  preserves the vertical structure if, given  a point $x$ in a vertical $C$ of a box $B$ and  a vertical $C'$ of a box $B'$ containing $f(x)$, then there is a clopen subset $\tilde{C} \subset C$ containing $x$ such that for every $y \in \tilde{C}$, $f(y) \in C'$.
\end{definition}
Alternatively, provided that $\Omega$ is minimal, a map $f \in  \Homeo_{\cL}(\Omega)$ preserves the vertical structure if and only if  its deplacement $\Phi$ is transversally locally constant. In the context of non-periodic repetitive tilings, this notion corresponds to the 
notion of \emph{strong pattern-equivariance} (see \cite{Ke}) of the map $t\mapsto \Phi_f(\omega-t)$ for any fixed $\omega \in \Omega$. We denote by $\Homeo_{vsp}(\Omega)$ the collection of homeomorphisms preserving the vertical structure. It is plain to check that $\Homeo_{vsp}(\Omega)$ is dense in $\Homeo(\Omega)$. We will denote by $D_{vsp}(\Omega)$ the path-connected component of the identity in $\Homeo_{vsp}(\Omega)$.


\section{Partition property}\label{sec:partition}
\begin{definition}
A group $G$ of homeomorphisms of $\Omega$ satisfies the {\em partition property} if for every box cover $\{B_{i}\}_{i=0}^{t}$ of $\Omega$, and for any $f \in G$, there exists a decomposition 
$f= g_{1}\cdots g_{\ell}$ where $g_{i} \in G$ and   $\operatorname{supp} g_{i} \subset B_{j(i)}$ for $i = 1, \ldots, \ell$.   
\end{definition}

In this section, following \cite{F} and using the box decomposition structure of tilable laminations we show:

\begin{prop}\label{prop:partition} Let $\Omega$ be a minimal  tilable lamination.
The two groups $ \operatorname{Homeo}_{\cL}^0(\Omega)$  and $ \operatorname{Homeo}_{vsp}^0(\Omega)$  satisfy the partition property.
\end{prop}

We will also show assertions (1) and (2) of Theorem \ref{teo:main}. To prove this result, we will use  several lemmas. We start by showing that every map having its  support included in a box of ball type is a deformation.
 \begin{lemma}\label{lem:isot-close-id}
 Let $\Omega$ be a minimal tilable lamination and $B$ be a box of ball type. Any map $g \in \operatorname{Homeo}_{\cL}(\Omega)$ (resp. in $\operatorname{Homeo}_{vsp}(\Omega)$), with support in the interior of $B$  is a deformation of the identity (resp. $g \in D_{vsp}(\Omega)$).  
 \end{lemma}
 \begin{proof}  The proof is classical by using the Alexander's trick. We can assume that the  closure $\overline{B}$ of the box  reads $h^{-1}(D \times C)$ in a chart $h$, with $D$ a closed  ball in $\R^{d}$ of radius $r>0$ centered at the origin. Since the support of the map $g$ is in $B$,  the map $g$ preserves any slice of the box $B$. 
So, for any $c \in C$, let $g_{c }\colon D \to D$ be the map  defined by   $g(h^{-1}(t,c)) =h^{-1}(g_{c}(t), c)$ for $t \in D$.  
Now, for any $t\in [0, 1]$, let $F_{t} \colon D \times C \to D \times C$ be the map 
$$F_{t}(x, c) =\begin{cases} ((1-t) g_{c}(\frac{x}{1-t}), c) & \textrm{ if } \vert x \vert < r(1-t)  \\
(x,c) & \textrm{ if }  \vert x \vert \ge r(1-t). \end{cases}
$$

It is plain to check the map $ h^{-1}F_{t} h$ gives an isotopy between the identity and the map $g$. 

In the case where $g \in   \operatorname{Homeo}_{vsp}(\Omega)$, up to subdivide the clopen set $C$, we can assume that the map $g_{c}$ is independent of $c$. Hence the isotopy is also in  $\operatorname{Homeo}_{vsp}(\Omega)$.
 \end{proof}

Given two subsets $A \subset B$  of a topological space $X$, an {\em embedding} $f \colon A \to B$ is a continuous and injective map. This embedding  is {\em proper} if $f^{-1}(\partial B) = A \cap \partial B$.   The next theorem says that any proper embedding of a neighborhood of a compact set $K$ into a ball, sufficiently close to the identity, can be isotoped to an embedding which is the identity on $K$. Moreover the isotopy depends continuously of the embedding. This theorem,  true in any dimension, generalizes a version of the {\em Schoenflies Theorem}.

\begin{teo}\label{teo:EK}\cite{EK} Let  $D$ be a (closed or open) ball in $\R^{d}$, $K \subset  D$ a compact subset and $ U $ a neighborhood of $K$ in $ D$. Then, for any proper embedding  $f \colon U \to D$   close enough of the identity (for the $C^{0}$ topology), there exists  a continuous map
$H \colon U \times [0,1] \to D$ such that: 
\begin{itemize}
\item For any $t\in [0,1]$, $H(\cdot, t) \colon U \to D$ is a proper embedding.
\item $H(\cdot, 0) = f(\cdot)$ and   $H(\cdot, 1)_{\vert K} = Id_{\vert K}.$
\item There is a compact  neighborhood $K_{2}$ of $K$ in $U$, such that for any $t\in [0,1]$, $H(\cdot, t)_{\vert U \setminus K_{2}} = f(\cdot)_{\vert U\setminus K_{2}}$.
\item $H$ depends continuously on $f$ for the $C^{0}$ topology. 
\end{itemize}
\end{teo}


 Applied in our context, a first consequence is that any map close to the identity can be interpolated by a map with a support in a box.
 \begin{lemma}\label{lem:2_7}
Let $\Omega$ be a tilable lamination, let $B$ be a box of ball type and  $B'$ be a box with closure included in $B$. Then there exists an $\varepsilon >0$ such that for any homeomorphism $f \in   \operatorname{Homeo}_{\cL}(\Omega)$ (resp.  $\operatorname{Homeo}_{vsp}(\Omega)$) with a displacement smaller than $\varepsilon$, there exists a map $g \in \operatorname{Homeo}_{\cL}(\Omega)$  (resp.  $\operatorname{Homeo}_{vsp}(\Omega)$) with $\supp g \subset B$ such  that $g_{\vert B'} = f_{\vert B'}$. 
Moreover, the displacement of $g$ depends continuously on $f$ for the $C^0$-topology.
\end{lemma}
\begin{proof} Without loss of generality, we may assume that $B$ reads $h^{-1} (D_{3} \times C)$ in a chart $h$ with $D_{3}$ a ball in $\R^{d}$. Assume that $D_{1}\times C$ is a compact neighborhood of $h(B')$ with  $D_{1} \subset  D_{3}$  a compact subset and let $D_{2} \subset D_{3}$ be  a neighborhood of $D_{1}$.   

We consider a map $f \in  \operatorname{Homeo}_{\cL}(\Omega)$ with a displacement smaller than $\varepsilon$ (defined later). By continuity, for a small enough $\varepsilon$, the set $f(h^{-1}(D_{2}\times C))$ is in $B$, and for any $c \in C$, $f(h^{-1}(D_{2} \times \{c\})) \subset  h^{-1}(D_{3} \times \{c\})$.
For any $c \in C$, let $f_{c} \colon D_{2} \to D_{3}$ be the embedding defined by  $f_{c} (\cdot) = (h f h^{-1})(\cdot, c)$ and let $K$ be a compact neighborhood of $\partial D_{2}$ proper  in $U =D_{2} \setminus D_{1}$.
For an $\varepsilon$ small enough and for any  $c\in C$,  Theorem \ref{teo:EK}  applied to the maps $f_{c} \colon U \to D_{3}$, gives us  embeddings $h_{c}  =H_{c}(\cdot, 1) \colon U \to D_{3}$. We define then the maps $\bar{f_{c}} \colon D_{3} \to D_{3}$ by $\bar{f_{c}}_{\vert U} := h_{c}$ and   $\bar{f_{c}}_{\vert D_{1} } :=f_{c \vert D_{1}}$ and  $\bar{f_{c}}_{\vert D_{3}\setminus D_{2}} :=Id_{\vert D_{3}\setminus D_{2}} $
Let $\bar{f} \colon B \to B$ be the map defined by $\bar{f} \circ h^{-1} (t,c) = h \circ (\bar{f_{c}}(t), c)$ for any $(t,c) \in D_{3} \times C$. By construction,  $\bar{f}$ is a homeomorphism, $\bar{f}_{\vert h^{-1}(D_{1} \times C) }    = {f}_{\vert h^{-1}(D_{1} \times C) }$ and $\bar{f}_{\vert \partial B} = Id_{\vert \partial B}$. So $\bar{f}$ can be extended  by the identity to all the tilable lamination $\Omega$ to define a homeomorphism. This gives the map $g$.      

Here again, when $f \in  \operatorname{Homeo}_{vsp}(\Omega)$, up to subdividing the clopen set $C$, we can assume that the map $f_{c}$ is independent of $c$. So the same is true for $\bar{f}$ and it belongs to  $\operatorname{Homeo}_{vsp}(\Omega)$. 
\end{proof}

\begin{prop}\label{prop:presque-partition}
Let $\Omega$ be a minimal tilable lamination and $\mathcal B = \{B_{i}\}_{i=1}^k$ be a box cover  of $\Omega$. Then, there are $\varepsilon >0$ and an integer $\ell>0$ such that for every $f \in \operatorname{Homeo}_{\cL}(\Omega)$ (resp. in  $\operatorname{Homeo}_{vsp}(\Omega)$) with displacement smaller than $\varepsilon$, there exists a decomposition $f = g_{1}\cdots g_{\ell}$ with $g_{i} \in \operatorname{Homeo}_{\cL}(\Omega)$ (resp.  $\operatorname{Homeo}_{vsp}(\Omega)$)   and $\supp g_{i} \subset B_{j(i)}$.
\end{prop}

\begin{proof} Let  $\cB' =\{B'_0,B'_1,\ldots,B'_m\}$ be the box decomposition of cubic type given by Lemma \ref{lem:BBG}. Up to subdividing  each box $B'_{i}$ into smaller boxes, we can assume that the closure of every  box $B'_{i}$ is included in a box $B_{j(i)}$. We will construct, by induction on $0 \le i \le d$, a homeomoprhism $f_{i}$ which equals the identity on a neighborhood of the $i$-skeleton of $\cB'$ and equals $f$ outside. At each step, we use Lemma \ref{lem:2_7} to approximate $f_{i-1}$ by maps with support in a small box.  

For any $0$-vertical boundary $V$ of a box $B'_{i}$, let $B^{(0)}_{V}$ be a box containing $V$ in its interior and included in a box $B_{j(i)}$.  Let $B^{(0)}_{1}, \ldots, B^{(0)}_{n}$ be the collection of  these boxes containing all the $0$-vertical boundaries. Up to refine the boxes $B^{(0)}_{i}$, we can assume that they are pairwise disjoint.  The union of all theses boxes $B^{(0)}_{i}$ covers the 0-skeleton of $\cB'$.

\emph{Step 0.}  Applying Lemma \ref{lem:2_7} to any box $B^{(0)}_{i}$  and  any neighborhood of the $0$-vertical $V \cap B^{(0)}_{i}$ (when not empty),  we get, for an $\varepsilon$ small enough, a $g_{i} \in \operatorname{Homeo}_{\cL}(\Omega)$ with $\supp g_{i} \subset B^{(0)}_{i}$ such that $g_{i}=f$ on a neighborhood of $V \cap B^{(0)}_{i}$.  It follows that the maps $g_{1}, \ldots, g_{n}$ commute;  and $ f_{0 } =g_{1}^{-1} \circ \cdots \circ g_{n}^{-1} \circ f$ is the identity in a neighborhood $U_{0}$ of the $0$-skeleton of $\cB'$. 
Moreover the displacement of $f_{0}$ continuously depends on the displacement of $f$.

\emph{ Step i.} $1 \le i \le d-1$. Let us assume that $f_{i} \in \operatorname{Homeo}_{\cL}(\Omega)$ equals the identity  on a neighborhood $U_{i-1}$ of the $i-1$-skeleton of $\cB'$.      We do the same as for the former step.  Let  $B^{(i)}_{1}, \ldots, B^{(i)}_{n_{i}}$ be a collection of boxes such that  any $i$-vertical boundary $V$ of  $\cB'$  is in the interior of a box $B^{(i)}_{j} \subset B_{t(j)}$. Up to  refine the boxes $B^{(i)}_{j}$, we may assume that the sets $B^{(i)}_{j} \setminus U_{i-1}$ are pairwise disjoint. 
Applying Lemma \ref{lem:2_7} to any box $B^{(i)}_{j}$  and  to a neighborhood  of the $i$-vertical $(V\setminus U_{i-1}) \cap B^{(i)}_{j}$ (when not empty),  we have, for an $\varepsilon$ small enough, a $g^{(i)}_{j} \in \operatorname{Homeo}_{\cL}(\Omega)$ with $\supp g^{(i)}_{j} \subset B^{(i)}_{j}$ such that $g^{(i)}_{j}=f$ on a neighborhood of $(V\setminus U_{i-1})  \cap B^{(i)}_{j}$.  We get that the maps $g^{(i)}_{1}, \ldots, g^{(i)}_{n_{i}}$ commute;  and $ f_{i } =(g_{1}^{(i)})^{-1} \circ \cdots \circ (g_{n_{i}}^{(i)})^{-1} \circ f_{i-1}$ is the identity in a neighborhood  of the $i$-skeleton.

Hence the homeomorphism $f_{d-1}$ preserves each box $B'_{i}$, and  $f_{d-1}$ can be written as  the composition of homeomorphisms with support in each box of the decomposition $\cB'$.   Moreover if $f \in  \operatorname{Homeo}_{vsp}(\Omega)$, then $f_{d-1} \in  \operatorname{Homeo}_{vsp}(\Omega)$ 
also. This proves the proposition.
\end{proof}

The following proposition shows the assertions (1) and (2) of Theorem \ref{teo:main}.
\begin{prop}\label{prop:connexite}
Let $\Omega$ be a minimal  tilable lamination. Then, $D(\Omega)$ is open in $\operatorname{Homeo}_{\cL}(\Omega)$ and 
$D(\Omega) = \operatorname{Homeo}_{\cL}^0(\Omega)$. 
Similarly, $D_{vsp}(\Omega)$ is open in  $\operatorname{Homeo}_{vsp}(\Omega)$ and $D_{vsp}(\Omega) = \operatorname{Homeo}_{vsp}^0(\Omega)$.
\end{prop}
\begin{proof} Let $\cB= \{B_{i}\}_{i=1}^{k}$ be a box cover of ball type of $\Omega$. Let $\rho$ be the Lebesgue number of $\cB$  and consider an atlas $\cB'=\{B'_{j}\}_{j=1}^{\ell}$  of $\Omega$ such that the diameter of any box $B'_{j}$ is smaller than $\rho$. Then by Proposition \ref{prop:presque-partition}, any map $f \in \operatorname{Homeo}_{\cL}(\Omega)$  with displacement small enough, can be written as a product of maps $g_{i} \in \operatorname{Homeo}_{\cL}(\Omega)$ with support in a $B'_{j(i)}$. Thus by Lemma \ref{lem:isot-close-id} we get  that any $g_{i}$ is in $D(\Omega)$, and finally $f \in D(\Omega)$. This means that the identity lies in the interior of $D(\Omega)$. Standard arguments on topological groups show then that  $D(\Omega)$ is open and closed in $\operatorname{Homeo}_{\cL}(\Omega)$ and $D(\Omega) = \operatorname{Homeo}_{\cL}^0(\Omega)$. The proof is similar for the group $\operatorname{Homeo}_{vsp}(\Omega)$.
\end{proof}

Finally, we can obtain the proof of Proposition \ref{prop:partition}.
\begin{proof}[of Proposition \ref{prop:partition}] Let $H$ be either  $ \operatorname{Homeo}_{\cL}^0(\Omega)$  or $\operatorname{Homeo}_{vsp}^0(\Omega)$ and let $\cB$ be an atlas of $\Omega$. Proposition \ref{prop:connexite} gives us a constant $\eta>0$ such that any element of $\operatorname{Homeo}_{\cL}(\Omega)$ (resp. in $\operatorname{Homeo}_{vsp}(\Omega)$) with a displacement smaller than $\eta$ is in $D(\Omega)$ (resp. $D_{vsp}(\Omega)$).  Up to refining the covering $\cB$, we can assume that every element of $\cB$  has a diameter smaller than $\eta$.
Since the group $H$ is connected, it is enough to show the partition property for any $f \in H$ with an arbitrary small displacement $\varepsilon$. By taking  the $\varepsilon$ given by  Proposition \ref{prop:presque-partition}, we can write $f$ as a product of elements  $g_{i}$ in $\operatorname{Homeo}_{\cL}(\Omega)$ with displacement smaller than $\eta$. By Proposition \ref{prop:connexite}, we get $g_{i} \in H$ and this shows the partition property. 
\end{proof}

\section{Simplicity of the commutator subgroup}\label{sec:simple}
Epstein's result \cite{E} asserts that if a group of homeomorphism is factorizable and acts transitively on open sets, 
then its commutator subgroup is simple. Let $\Omega$ be a minimal tilable lamination and let $G$ 
be a subgoup of $\Homeo_{\cL}^{0}(\Omega)$. In general, we cannot expect $G$ to act transively on open sets. 
We need to replace the transivity condition with another one which is more adapted to laminated spaces. Thus, we give here, a sufficient condition on a  subgroup of $\operatorname{Homeo}_{\cL}^{0}(\Omega)$ so that the derived subgroup is simple. 
\begin{teo}\label{teo:simple}
Let $\Omega$ be a minimal $\R^{d}$ tilable lamination. Let $G \subset  \operatorname{Homeo}_{\cL}^0(\Omega)$ be a group such that:
\begin{itemize}
\item[i)] $G$ satisfies the partition property.
\item[ii)] For any boxes $B_{1} =C -D_{1} , B_{2} = C-D_{2}$ of ball type whose closures lie in a box $B =C-K$, there exists a $g \in G'$ such that $B_{2} \subset g(B_{1})$. 
\end{itemize}
Then the derived group $G'=[G,G]$ is simple. 
\end{teo}
\begin{proof} Let $N$ be a non-trivial normal subgroup of $G'$. We have to show that $N =G'$. 

\begin{lemma}\label{l.good_cover}
There exists an atlas $\cB$ of the solenoid $\Omega$ such that for every box $B$ of $\cB$ there is 
a map $n_{B} \in N$ such that ${B} $ and $n_{B}({B})$ are disjoint.  
\end{lemma}

\begin{proof} Let $Id\neq n\in N $. There is a box $B_{0}= C_{0}-D_{0}$ of ball type such that $B_{0}$ and $n(B_{0})$ are disjoint.
Since the translation flow is free and minimal and translations have the vertical structure preserving property, it follows that  for every  $x \in \Omega$, there is a clopen subset $C_{x} \subset C_{0}$ and an open  ball $D_{x}\subset \R^{d}$ with $D_{0} \subset D_{x}$ such that $C_{x}-D_{x}$ is a box of $\Omega$ containing the point $x$. 

Let $\tilde{B}_{x}$ be a box included in $C_{x}-D_{x}$ containing the point $x$.  By hypothesis ii), there is a $g \in G'$ such that $\tilde{B}_{x} \subset g(C_{x}-D_{0})$. It is then straightforward to check that the box $\tilde{B}_{x}$ is disjoint from its image by the  map $g \circ n \circ g^{-1} \in N$. 
The collection of boxes $\{\tilde{B}_{x}\}_{x\in \Omega}$ satisfies the condition of the statement.  
\end{proof}


Let $\cB$ be the finite cover given by Lemma \ref{l.good_cover} of the tilable lamination $\Omega$ by boxes and let $\rho>0$ be its Lebesgue number. Let us recall that  for any ball of radius $\rho$ in  $ \Omega$ there exists a box of $\cB$ containing this ball. Let $\cB_{1}$ be a box cover of cubic type of  $\Omega$,  equivalent to $\cB$, and such that any box has a diameter smaller than $\rho$. It follows that when two boxes  $B_{1}, B_{2}$ of $\cB_{1}$ are intersecting, there exists a box $B$ of $\cB$ containing  $B_{1}\cup B_{2}$. 


The following  is an algebraic lemma due to T. Tsuboi \cite{Ts2}.
\begin{lemma}[{\cite[Lemma 3.1]{Ts2}}]\label{lem:Ts2}
Let $B$ be a box  and  $n$ be an homeomorphism such that $n(B) \cap B =\emptyset$. Then for any  homeomorphisms $a,b \in G$ with supports in  $ B$, the commutator $[a,b]$ can be written as a product of 4 conjugates of $n$ and $n^{-1}$.
\end{lemma}
\begin{proof}
Let $h = n^{-1} a n$, since the supports are disjoint, we have $hb = bh$. So we get
\begin{align*}
aba^{-1}b^{-1} &= nh n^{-1}bn h^{-1} n^{-1}b^{-1}\\
& =  nh n^{-1} h^{-1} h b n h^{-1}b^{-1}bn^{-1} b^{-1}\\
 &= n (h n^{-1}h^{-1}) (b h n h^{-1} b^{-1}) (b n^{-1} b^{-1}).
\end{align*}  
\end{proof}

Now, for each $B\in \mathcal{B}_{1}$, let $G_B$ be the subgroup of $G$ of homeomorphisms with support in $B$, and let $H$ be the subgroup of $G$ generated by all the $G_B$ with $B\in\mathcal{B}_{1}$. By the partition property (item i)),  the groups $H$ and $G$ are the same. It is well-known that the commutator subgroup of $H$ is generated by the conjugates of commutators of elements in a generating set of $H$. So to prove the theorem, we just have to show for any boxes  $B_1, B_2$  in $\mathcal{B}_{1}$, and for   $f_1 \in G_{B_{1}}$ and  $f_2 \in G_{B_{2}}$, that the commutator $[f_1,f_2]$ belongs to $N$.

If the boxes $B_1$ and $B_2$ do not intersect, then every point of $\Omega$ is fixed by either $f_1$ or $f_2$, which means that $[f_1,f_2]=id$ and thus belong to $N$.

Suppose now that $B_1$ and $B_2$ intersect and  let  $B$ be a box in $\cB$ containing $B_1\cup B_2$. 
By Lemma \ref{l.good_cover}, there exists a $n \in N$ such that $B \cap n(B) =\emptyset$. Thus by Lemma \ref{lem:Ts2}, we have $[f_{1}, f_{2}] \in N$, and then $H'= N$. 
\end{proof}

\section{Perfectness}\label{sec:perfect}

To show that  the groups $\operatorname{Homeo}_{\cL}^{0}(\Omega)$ and 
$\operatorname{Homeo}_{s}^{0}(\Omega)$ are simple, we will first prove in this section they are perfect. For this, we need the next  lemma stating  a transitivity of the action of the group $D_{vsp}(\Omega)$ on specific boxes of a same  box. This is a reinforcement  of condition ii) in Theorem \ref{teo:simple}. We will deduce then the perfectness. This will imply, together with the partition property, that these groups satisfy the conditions of Theorem \ref{teo:simple}, and henceforth they are simple. 

\begin{lemma}\label{l.homeo_boxes}
Let $B_{1}=C-D_{1}$ and $B_{2}=C-D_{2}$ be two boxes of ball type, and let $B_{0} = C-V$ be a box containing the closures of $B_{1}$ and $B_{2}$.  Then there exists a $g \in D_{vsp}(\Omega)$  such that $B_{2}= g(B_{1})$.  
\end{lemma}
\begin{proof} Let $h$ be the chart associated to the box $C-V$. The boxes $B_{1}$ and $B_{2}$ read respectively $h^{-1}(D_{1}\times C)$ and $h^{-1}(D_{2} \times C)$, with $D_{1}, D_{2}$ two balls in $V$. Up to composing  with a translation $T_{\rho}$, we may assume that $D_{1}\subset D_{2} \subset V$ or  $D_{2}\subset D_{1} \subset V$. In both cases, it is straightforward to construct a homeomorphism $\psi \in D_{vsp}(\Omega)$ with support in $C-V$ such that $\psi(h^{-1}(D_{1}) \times C) )= h^{-1}(D_{2}\times C)$. \end{proof}

The proof of the next theorem follows directly the ideas of Fr\'ed\'eric Le Roux (see Theorem 1.1.3 in  \cite{B} for a proof on a surface). On a manifold, this shows directly that the group of homeomorphism is simple. Here, because of the lack of homogeneity, it enables just to show the groups $\operatorname{Homeo}_{\cL}^0(\Omega)$  and $\operatorname{Homeo}_{vsp}^0(\Omega)$ are perfect. 

\begin{teo}\label{teo:perfectness}
 For $\Omega$ a  minimal $\R^d$-tilable lamination, the groups   $\operatorname{Homeo}_{\cL}^0(\Omega)$  and  $\operatorname{Homeo}_{vsp}^0(\Omega)$ are perfect. 
\end{teo}
\begin{proof}
Let $H$ denotes either  $\operatorname{Homeo}_{\cL}^0(\Omega)$  or $\operatorname{Homeo}_{vsp}^0(\Omega)$. We have to show that any element of $H$ can be written as a commutators product. It is simple to find two non commuting  elements $a, b \in D_{vsp}(\Omega)$ with supports in a box $B \subset \Omega$ of ball type. So the element $g =[a,b] \in H$ is not the identity. We will show that $N(g)$, the normal subgroup generated by $g$, contains all the elements $f$  of $H$ with support in the box $B$. Recall that a conjugate of a commutator is still a commutator, it will follow that $f$ can be written as a finite product of commutators. Since the box $B$ is arbitrary, we get the conclusion by  the partition property (Proposition \ref{prop:partition}).

We have  $g\neq Id$, so we  consider a box $B' \subset B$ such that $g(B')$ and $B'$ are disjoint.  In a chart $h$, we may assume that  $B'$  reads $h^{-1} (B_{r}(0) \times C)$ where $B_{r}(0)$ denotes the Euclidean ball in $\R^d$ of radius $r>0$ centered at $0$.  For any integer $n \ge 0$, we define a nested sequence of boxes  $B_{n}:= h^{-1}(B_{r/2^{n+1}}(0) \times C)$.   
It is simple to construct   an element $\psi$ of $D_{vsp}(\Omega)$ with support in $B'$, such that $\psi(B_{n}) = B_{n+1}$ of $n \ge 0$.  We get then that the homeomorphism $k= [\psi, g] \in N(g)$ satisfies $ k(B_{n}) = B_{n+1}$  for $n \ge 0$ and $\supp k \subset B' \cup g(B')$ ($k$ is the product of $\psi$ and $g\psi^{-1}g^{-1}$that have disjoint supports).

Let $A_{n} = B_{n} \setminus B_{n+1}$ and let us show that $N(g)$ contains all the element of $H$ with support in $A_{1}$.  For any $\phi \in H$ with a support in $A_{1}$,  we claim  that $\phi k$ and $k$ are conjugate:  notice, we have  $ k = \phi^{-1} (\phi k) \phi$  on $A_{0}$,  $\bigsqcup_{n \ge 0} A_{n} = \bigsqcup_{n\ge 0} g^{n}(A_{0}) = \bigsqcup_{n \ge 0} (\phi k)^{n} (A_{0})$ and $k_{\vert B \setminus B_{0}} = \phi k_{\vert B \setminus B_{0}}$; It is then standard to check that the continuous homeomorphism  $\tilde{\phi}$ defined by  
$$ \tilde{\phi}_{\vert A_{n}} := (\phi k)^{n} \phi k^{-n}_{\vert A_{n}}  \textrm{ for any }  n\ge 0 \hspace{0.5cm} \textrm{ and } \hspace{0.5cm} \tilde{\phi}_{\vert \Omega \setminus B_{0}} = Id,$$  can be extended by continuity to $\overline{\bigcup_{n \ge 0} A_{n}} = B_{0}$, is in $H$ and satisfies $k = \tilde{\phi}^{-1} (\phi k) \tilde{\phi}$  on  $\Omega$. 
We get then $\phi k \in N(k) \subset N(g)$, so $\phi \in N(g)$.

Finally, let $f \in H$ with a support in $B$. By Lemma \ref{l.homeo_boxes}, there exists a $\phi \in H$ such that the support of $\phi f \phi^{-1}$ is in $A_{1}$. So by the last result we have $f \in N(g)$.  
\end{proof}

We have then the groups $\Homeo_{\cL}^0(\Omega)$ and  $\Homeo_{vsp}^0(\Omega)$ equal their commutator groups. So by  Lemma \ref{l.homeo_boxes} and Theorem \ref{teo:simple}, we get  the main  result: Theorem \ref{teo:main}.

\section{Uniform perfectness in dimension one} 

Theorem \ref{teo:perfectness} asserts that any homeomorphism of a tilable lamination $\Omega$ is a product of commutators. For the  one dimension, we can  be more precise. 
\begin{teo}\label{teo:uniform perfectness} For $\Omega$ a  minimal $\R$-tilable lamination, any element of   $\Homeo_{\cL}^0(\Omega)$  (resp. $D_{vsp}(\Omega)$) can be written as a product of two commutators of   $ \Homeo_{\cL}^0(\Omega)$  (resp.  $D_{vsp}(\Omega)$).
\end{teo}

Before proving this theorem, we need some technical lemmas. 
The first one solves the problem of perfectness for homeomorphisms with support in a box.
Recall that any map $f \in \Homeo_{\cL}^{0}(\Omega)$ preserves the orientation, so, if $\Phi$ denotes its displacement, for any $\omega \in \Omega$, the map $\R \ni t \mapsto t+ \Phi(\omega-t) \in \R$ is increasing. 

\begin{lemma}\label{lem:cuadrado} Let $\Omega$ be a  minimal $\R$-tilable lamination and  let $f \in  \Homeo_{\cL}^0(\Omega)$  (resp.  $\Homeo_{vsp}^0(\Omega)$) with support included in a box $B$.   Then  there exists a homeomorphism $g \in \Homeo_{\cL}^0(\Omega)$  (resp.  $\Homeo_{vsp}^0(\Omega)$) with support in $B$ such that
$gfg^{-1} = f^{2}$.  In particular $f = [ g, f]$.
\end{lemma}
\begin{proof}  Since $f$ preserves the orientation,  it preserves each slice of the box $B=C-I$, with $C$ a clopen set and $I$ an interval.  For any $x \in C$,  we denote by $f_{x} \colon I \to I$ the increasing map induced on the slice of $x$: i.e.  defined by $f_{x}(t)  = t+ \Phi(x-t)$ so that $f(x-t) = x- f_{x}(t)$ for any $t\in I$.

For any $z_{0} \in \{z\in B; f \neq Id \}$, let the vertical $C_{z_{0}}$ be $C- t_{0}$ where $z_{0}$ writes $x_{0}-t_{0}$ with $x_{0}\in C$, $t_{0}\in I$. We define the local strip 
$$V_{z_{0}}:= \{x-t; \ x\in C, t_{0} \le t < f_{x}(t_{0}) \}.$$ 
By the definition, the sets $\{ f^n(V_{z_{0}})= V_{f^n(z_{0})}\}_{n\in \Z}$ are pairwise disjoints, and $\cup_{n\in \Z}f^n(V_{z_{0}})$ is a $f$-invariant open  set. Hence there exist a  collection at most countable of  points $\{z_{n}\}_{n\ge 0}\subset \{f \neq Id \}$ and  local strips $V_{n} =V_{z_{n}}$ such that 
$$ \supp f = \cup_{n\ge 0} \overline{\cup_{p\in \Z}f^p(V_{n})} \textrm{ and the sets  } \{{{\cup_{p\in \Z}f^p(V_{n})}}\}_{n\ge 0} \textrm{ are pairwise disjoint}.$$
Notice that since $f$ preserves the orientation, a point is fixed by $f$ if and only it is a fixed point of $f^2$. So we have $\supp f= \supp f^2$.

For each $n\ge 0$, $z_{n } =x_{n}-t_{n}$ with $t_{n}\in I$, $x_{n}\in C$, let $h_{n} \colon [t_{n}, f_{z_{n}}(t_{n})) \to [t_{n}, f^2_{z_{n}}(t_{n}))$ be the bijective affine map fixing $t_{n}$. It is then straightforward to check that the continuous map $g_{n}$ defined on  $\cup_{p\in \Z}f^p(V_{n})$ by
$$g_{n \vert f^p(V_{n})} := f^{2p} \circ h_{n} \circ f^{-p},$$ 
can be continuously extended by the identity  to $\partial{\cup_{p\in \Z}f^p(V_{n})}$ and satisfies $f^2 \circ  g_{n} = g_{n} \circ f$ where it is defined. Hence we can define a homeomorphism $g$ on  $\Omega$ with support in $B$  such that $g_{ \vert {\cup_{p\in \Z}f^p(V_{n})}} = g_{n}$ for every $n \ge 0$.
Notice  furthermore that $g$ is in $\operatorname{Homeo}_{vsp}^0(\Omega)$ when  $f$ is. 
\end{proof}

The next lemma, is a version of lemma \ref{lem:2_7} without the condition to be close of the identity.
\begin{lemma}\label{lem:2_7bis}Let $\Omega$ be a minimal $\R$-tilable lamination and let $f \in   \operatorname{Homeo}_{\cL}^0(\Omega)$  (resp.  $\operatorname{Homeo}_{vsp}^0(\Omega)$).  Suppose that $B'=C-J$ and $B=C-I$ are boxes of cubic type such that the closure of $f(B')\cup B'$ is contained in $B$. Then, there exists a $g$ in  $D(\Omega)$  (resp.  $D_{vsp}(\Omega)$)  with support  contained in $B$ such that $f|_{B'}=g|_{B'}$. 
\end{lemma} 
\begin{proof}
Without loss of generality, we may assume that  $I,J$ are two open intervals such that $0\in \overline{J}\subset {I}$.  Consider $\eta:I\rightarrow[0,1]$ a continuous function that is equal to zero on the boundary of $I$, is equal to one on $J$ and affine on each component of $I\setminus J$. Let $\phi$ be the displacement function of $f$, and define $\psi \colon \Omega \to \R$ by 
 \[\psi (x-t) =  \eta(t)\phi(x-t)  \quad  \textrm{ for any }x\in C, t\in I,\]
and by zero on the complement of $B$. It is clear that $\psi$ is a continuous function. Thus, $g(x) := x - \psi(x)$ is continuous and coincides with $f$ on $B'$, and since $J\subset I$, it is also increasing by the choice of $\eta$. It is plain to check $g \in D(\Omega)$.
\end{proof}

\begin{lemma}\label{lem:scindement}
Let $\Omega$ ba a minimal $\R$-tilable lamination and let $f \in \Homeo_{\cL}^0(\Omega)$  (resp. $\Homeo_{vsp}^0(\Omega)$). Then there exist two boxes $B' \subset B$ and two homeomorphisms  $f_{1}, f_{2} \in \Homeo_{\cL}^0(\Omega)$ (resp. $\Homeo_{vsp}^0(\Omega)$) such that 
\begin{itemize}
\item $\operatorname{supp} f_{2} \subset B$;
\item $f_{1} \vert_{ B'} = Id \vert_{B'}$;
\item $f= f_{1} \circ f_{2}$.
\end{itemize}
\end{lemma} 
\begin{proof}
Let $x$ be a point of $\Omega$. The points $x$ and $f(x)$ are in the same leaf, so they belong to a same box $B =C-I$ of cubic type. By continuity, there exists a small box $x\in B'=C-J$ such that  the closures of $B', f(B')$ are in $B$.  Let $f_{2}$ be the map given by Lemma \ref{lem:2_7bis}, and let $f_{1} = f \circ f_{2}^{-1}$.  It is straightforward to check they satisfy the conditions of the lemma.
\end{proof}

Next we need a topological lemma on one dimensional tilable laminations. If $B=C-(a,b)$ is a box of cubic type, for an element  $x\in C-b$, its \emph{return time to} $C-a$ is
$$\tau_{C-a}(x) = \inf \{t >0; \ x-t \in C-a \}.$$
By minimality, $\tau_{C-a}(x)$ is finite for any $x\in C-b$, and the map $\tau_{C-a} \colon C-b \to \R$ is locally constant, hence continuous. 

\begin{lemma}\label{lem:topology} Let $\Omega$ be a $\R$-tilable lamination, and let $B =C-(a,b)$ be a box of cubic type. Then the following map is an homeomorphsim.
 $$ \begin{array}{ccc} 
      \{(x,t); \ x \in C-b, \ 0\le t  \le \tau_{C-a}(x) \} &  \longrightarrow  &  \Omega \setminus B\\
      (x,t) & \mapsto & x-t.\\
   \end{array}$$
\end{lemma}
The proof is plain.

\begin{proof}[of Theorem \ref{teo:uniform perfectness}]  Let us denote  by $H$ the group $ \Homeo_{\cL}^0(\Omega)$  or $\Homeo_{vsp}^0(\Omega)$ and let $f \in H$. 
 Let $f_{1}$ and $f_{2}$ be the homeomorphims in $H$ and $B, B'$ be the boxes   given by Lemma \ref{lem:scindement}.  From Lemma \ref{lem:topology} applied to the box $B' = C-(a,b)$, and since the map $\tau_{C-a}$ is locally constant, there exists a clopen partition $\{C_{1}, \ldots, C_{\ell} \}$ of $C$ such that for any $i$,  $\tau_{C-a} \vert C_{i}$ is constant, equals to $\tau_{i}$ and $\{C_{i}-[0, \tau_{i}] \}_{i=1}^{\ell}$ is a covering of $\Omega \setminus B$ by closed boxes with interior pairwise disjoint.

Hence, the map $f_{1}$ preserves any box $C_{i} -[0, \tau_{i}] $, so it  can be written as a product of maps $g_{1}  \cdots  g_{\ell}$, where any $g_{i} \in H$ and $\operatorname{supp} g_{i} \subset C_{i}-[0, \tau_{i}]$. 
By Lemma \ref{lem:cuadrado},  $f_{2}$ is a commutator and any $g_{i}$ is a commutator  $[a_{i}, b_{i}]$ where the homeomorphisms $a_{i}, b_{i}\in H$ have their support in the box $C_{i}-[0, \tau_{i}]$. Since two homeomorphisms  with  disjoint interior of supports   commute, we have 
\begin{align*}
f_{1 }=\prod_{i=1}^{\ell} g_{i} = \prod_{i=1}^{\ell} [a_{i}, b_{i}] = [\prod_{i=1}^{\ell} a_{i}, \prod_{i=1}^{\ell} b_{i} ].
\end{align*}
It follows that $f$ may be written as a product of two commutators.
\end{proof}

\textbf{Acknowledgments.}   It is a pleasure for S. Petite to acknowledge A. Rivi\`ere for all the discussions on the subtleties of the Schoenflies Theorem.

\address{
   Jos\'e Aliste-Prieto\\
   Departamento de Matematicas\\
   Universidad Andres Bello\\
   Republica 220, Santiago, Chile
   \email{jose.aliste@unab.cl}}

\address{
   Samuel Petite\\
  Laboratoire Ami\'enois de\\ Math\'ematique Fondamentale et Appliqu\'ee,\\ CNRS-UMR 7352,\\
Universit\'{e} de Picardie Jules Verne,\\
33 rue Saint Leu,\\
80039
Amiens Cedex, France.
\email{samuel.petite@u-picardie.fr}}
\end{document}